\documentclass{amsart}
\usepackage[utf8]{inputenc}
\usepackage{amsfonts}
\usepackage{amsmath}
\usepackage{amssymb}
\usepackage{amsthm}
\usepackage{hyperref}
\usepackage[english]{babel}
\usepackage{url}
\usepackage{verbatim,listings}
\usepackage{tikz,tikz-cd}
\usepackage[margin=1.0in]{geometry}

\newtheorem{thm}{Theorem}[section]
\newtheorem{lem}[thm]{Lemma}

\theoremstyle{definition}

\theoremstyle{remark}
\newtheorem*{rem}{Remark}

\newcommand{\divides}{\,\Big{|}\,}

\title{families of congruences for partitions with $k$-colored odd parts}
\author{Samuel Wilson\\ University of Tennessee - Knoxville} 
\date{}

\begin{document}

\begin{abstract}
    The study of integer partitions and their congruences dates back to 1919 when Ramanujan discovered his famous congruences for the partition function, $p(n)$. Since then, many other kinds of partition functions have been discovered, as well as their respective congruences. Recently, Hirschorn and Sellers have consider partitions in which the odd parts may appear in $k$ colors and the even parts are restricted to at most one color. It turns out that these partitions exhibit fascinating families of congruences. In this paper, we look at a set of congruences that give rise to infinite families modulo 3. We also give some questions at the end that could aid further research into these partitions.
\end{abstract}

\maketitle

\section{Introduction}
Integer partitions and their congruences have been studied in Number Theory and Combinatorics for over a century. An integer partition (or simply, a partition) of $n$, is a non-increasing sequence of positive integers $\lambda_1 \geq \lambda_2\geq\ldots\geq\lambda_r \geq 1$  which sum to $n$. We call each $\lambda_i$ a \emph{part} of the partition. We denote the number of partitions of $n$ as $p(n)$, whose generating function is given by
$$
\sum_{n\geq 0}p(n)q^n = \prod_{n\geq 1}\frac{1}{(1-q^n)} = \frac{1}{f_1},
$$
where $q := e^{2\pi i z}$ and 
$$
f_k := \prod_{n\ge1} (1-q^{kn}).
$$
In recent years, much attention has been given to multicolored partitions. We denote the number of $k$-colored partitions of $n$ as $p_k(n)$, which gives the number of ways $n$ can be written as a partition in which each part can be given one of $k$-colors. The generating function for $p_k(n)$ is given by 
$$
\sum_{n\geq 0}p_k(n)q^n = \prod_{n\geq 1}\frac{1}{(1-q^n)^k} = \frac{1}{f_1^k}.
$$

Recently, there has been much interest in partitions in which the even parts appear in at most one color, and the odd parts may appear in at most $k$-colors. These have been studied by the likes of Sellers, Hirschhorn, Thejitha, and Fathima (see
\cite{hirschhorn2025familycongruencesmodulo7},\cite{sellers2025elementaryproofsgeneralizationsrecent},\cite{thejitha2025arithmeticpropertiespartitions1colored}, and \cite{thejitha2026arithmeticpropertiescoloredpartitions}). The generating function for these partitions is given by 
$$
\sum_{n\geq0}a_k(n)q^n = \prod_{n\geq 1}\frac{(1-q^{2n})^{k-1}}{(1-q^n)^k} = \frac{f_2^{k-1}}{f_1^k}.
$$
Multiple congruences have been found for the partition numbers $a_k(n)$. For example, Sellers proved the following lemma in \cite{sellers2025elementaryproofsgeneralizationsrecent} that we will use later.
\begin{lem}\label{sellerslemma}
    For $j \geq 0, 0\leq t \leq 8,$ and all $n \geq 0$,
    $$
    a_{27j+3t+2}(27n+(18+t)) \equiv 0 \pmod{3}
    $$
\end{lem}
The utility of this lemma lies in its ability to provide base cases for infinite families of congruences mod 3. The following example of which was proven by Thejitha and Fathima in \cite{thejitha2026arithmeticpropertiescoloredpartitions}.
\begin{thm}\label{mehjinlemma}
    For all $n,k\geq0$,
    $$
    a_5\left(3^{2k+3}n+\frac{153(3^{2k})-1}{8}\right) \equiv 0 \pmod{3}
    $$
\end{thm}
Proofs similar to Theorem \ref{mehjinlemma} usually have the same structure. Namely, the infinite families arise from two phenomena: a base congruence (for example, Lemma \ref{sellerslemma}) and an internal congruence that ``lifts" the base congruence to higher powers of the modulus. This is the technique we will be employing in order to prove our main theorem, which generalizes the results of Thejitha, Fathima, and Sellers.
\begin{thm}\label{mainthm}
    Let $\alpha \in \{1,3,4,6,7,9,11,12,14,15,20,22,23,27,28,30,31,36,38,39,46,47,54,55,63\}$ 
    and write $\alpha= 9j+t$, $0\leq t \leq 8$. Then for all $n,k \geq 0$,
    $$
    a_{3\alpha+2}\left(3^{2k+3}n+9^k(18+t)+\alpha\frac{9^k-1}{8}\right) \equiv 0 \pmod{3}.
    $$
\end{thm}

\begin{rem}
    The above family of values for $\alpha$ may seem a bit odd or incomplete. Indeed, it turns out other congruences can likely be found for many if not all of the $\alpha \ge 0$. The values above are the only ones that fall into the infinite family described. It is possible that other families of congruences exist for many missing values of $\alpha$.
\end{rem}
The structure of the paper is as follows. In Section 2 we recall some necessary facts about modular forms and eta quotients. In Section 3, we prove 2 lemmas which allow us to prove our main theorem by simply checking Sturm bounds, and in Section 4 we prove some Ramanujan congruences of $a_{11}(n)$, and conclude with some discussion on future problems related to these partition numbers.

\section{Prerequisites}
In order to prove our main theorem, we will utilize the theory of modular forms on congruence subgroups. We will recall some basic facts about the theory, but for a more in-depth look, one can take a look at \cite{webofmod},\cite{Diamond_Shurman_2016}, and \cite{Kilford_2015}. To that end, let 
$$
\Gamma_0(N) = \left\{
\begin{pmatrix}
a & b \\
c & d
\end{pmatrix}
\in \mathrm{SL}_2(\mathbb{Z}) : c \equiv 0 \pmod{N} \right\}.
$$
We denote the space of holomorphic modular forms of weight $k$, level $N$, and character $\chi$ as $M_k(\Gamma_0(N),\chi)$.
Additionally for $\Gamma_0(N)$ and integers $n,k \ge 1$, we define linear operators
$$
T_n : M_k(\Gamma_0(N),\chi) \to M_k(\Gamma_0(N),\chi),
$$
called \emph{Hecke operators}. These operators act on Fourier expansions by
$$
\left( \sum_{m\geq 0} a(m) q^m \right) \Big| T_n = \sum_{m\geq 0} \left( \sum_{d \mid (m,n)}\chi(d) d^{k-1} a(mn/d^2) \right) q^m.
$$
 If $n=p$ is prime, then we have 
$$
\left( \sum_{m\geq 0} a(m) q^m \right) \Big| T_p = \sum_{m\geq 0}\left( a(pm)+\chi(p)p^{k-1}a(m/p)\right)  q^m.
$$
Note that, mod $p$, we have 
$$
\left( \sum_{m\geq 0} a(m) q^m \right) \Big| T_p \equiv \sum_{m\geq0} a(pm)q^m \pmod{p}.
$$
We will also use $T_p^k$ to notate applying $T_p$ $k$-times. \\

We note that the normalized weight 4 Eisenstein series of level 1 is given by
$$
E_4(z) = 1+240\sum_{n\geq1}\sum_{\substack{d\geq 0 \\ d|n}}d^3q^n.
$$

We also define the Dedekind eta function to be the modular form of weight $1/2$ given by
$$
\eta(z) = q^{\frac{1}{24}}\prod_{n \geq 1} (1-q^n).
$$
For the purposes of this paper, we are primarily concerned with \emph{eta-quotients}, which are functions of the form 
$$
f(z) = \prod_{\delta |N}\eta(\delta z)^{r_\delta}.
$$
It is well known that the following conditions indicate whenever an eta-quotient is a modular form.
\begin{lem}[Theorem 1.64 and 1.65 of \cite{webofmod}]\label{1.64 1.65}
    Suppose $f(z) = \prod_{\delta |N}\eta(\delta z)^{r_\delta}$ is an eta quotient with $k=\frac{1}{2}\sum_{\delta|N}r_\delta \in \mathbb{Z}$ which satisfies
    $$
    \sum_{\delta |N}\delta r_\delta \equiv 0 \pmod{24},
    $$
    $$
    \sum_{\delta|N}\frac{N}{\delta}r_\delta \equiv 0 \pmod{24},
    $$
and, for any cusp $\frac{c}{d}$ of $\Gamma_0(N)$ with $c,d\geq 1$ and relatively prime,
    $$
    \sum_{\delta|N}\frac{\gcd(d,\delta)^2r_\delta}{\delta} \geq 0.
    $$
Then $f \in M_k(\Gamma_0(N),\chi)$, where 
$$
\chi(d) := \displaystyle\left(\frac{(-1)^k\prod_{\delta|N} \delta^{r_\delta}}{d}\right).
$$. 
\end{lem}
Finally, we end our prerequisites with a theorem of Sturm that will allow us to prove the necessary congruences. Define $ord_p(f)$ as the smallest $n$ such that the coefficient of $q^n$ is not $0$ mod $p$.
\begin{lem}[Theorem 2.58 of \cite{webofmod}]\label{sturmbound}
    Let $p$ be prime and suppose $f(z), g(z) \in M_k(\Gamma_0(N),\chi)$ have integer coefficients. If
    $$
    ord_p(f(z)-g(z)) > \frac{k}{12}[\mathrm{SL}_2(\mathbb{Z}):\Gamma_0(N)],
    $$
    then $$ord_p(f(z)-g(z)) = \infty.$$ In other words, the coefficients of their Fourier expansions are congruent mod $p$.
\end{lem}

We are now equipped to prove our necessary lemmas and main theorem.

\section{Proof of Theorem \ref{mainthm}}
We begin by proving 2 lemmas. The first will show that, given a base congruence and an internal congruence, we have an infinite family of congruences. The second will establish the necessary conditions for the existence of these internal congruences. Then we prove our main theorem by using Lemma \ref{sturmbound}.

\begin{lem}\label{mainlemma}
    Fix $\alpha \geq 0$ such that $\alpha = 9j+t$, where $0 \leq t \leq8$. For all $n \geq 0$,
    if 
    $$
    a_{3\alpha+2}(27n + (18+t)) \equiv 0 \pmod{3}, 
    $$
    and
    $$
    a_{3\alpha+2}(9n+t) \equiv a_{3\alpha+2}(81n+10t+9j) \pmod{3},
    $$
    then
    $$
    a_{3\alpha+2}\!\left(3^{\,2k+3} n + \delta_k\right) \equiv 0 \pmod{3}
    $$
    where
    $$
    \delta_k = 9^k(18+t) + \alpha \frac{9^k - 1}{8}.
    $$
\end{lem}
\begin{proof}
We proceed by induction. \\

\textbf{Base case} ($k=0$).
Since we have $\delta_0 = 18+t$ and $3^{2\cdot0+3}=27$, we obtain
$$
a_{3\alpha+2}(27n+(18+t)) \equiv 0 \pmod{3},
$$ \\
which follows by our hypothesis. \\

\textbf{Inductive Step}. Assume our hypothesis holds up to $k$. Note that 
\begin{align*}
     &a_{3\alpha+2}\left(3^{2(k+1)+3}n+9^{k+1}(18+t)+\alpha\frac{9^{k+1}-1}{8}\right) \\
    =\,\, &a_{3\alpha+2}\left( 9 \left(3^{2k+3}n+9^k(18+t)+\alpha\frac{9^k-1}{8}\right)+\alpha \right)
\end{align*}
Now, by our internal congruence,
\begin{align*}
    &a_{3\alpha+2}(9n+t) \equiv a_{3\alpha+2}(9(9n+t)+9j+t) \pmod{3} \\
    \implies &a_{3\alpha+2}(9n+t) \equiv a_{3\alpha+2}(9(9n+t)+\alpha) \pmod{3}, \\
\end{align*}
which implies 
$$
a_{3\alpha+2}(m) \equiv a_{3\alpha+2}(9m+\alpha) \pmod{3},
$$
whenever $m \equiv t \pmod{9}$. Therefore, we have 
\begin{align*}
    &a_{3\alpha+2}\left( 9 \left(3^{2k+3}n+9^k(18+t)+\alpha\frac{9^k-1}{8}\right)+\alpha \right)  \\
    \equiv\,\, &a_{3\alpha+2}\left(3^{2k+3}n+9^k(18+t)+\alpha\frac{9^k-1}{8}\right) \equiv 0  \pmod{3},
\end{align*}
which completes the proof.
\end{proof}

\begin{lem}\label{neccessarymodforms}
    Fix $\alpha \geq 0$ such that $\alpha = 9j+t$ where $0\leq t \leq 8$, $r\in\{0,1\}$ such that $r \equiv\alpha+1 \pmod 2$, and $0 \leq A \leq 24$ such that $A \equiv -3t-3j-2r \pmod{24}$. Then
    $$
    g_1(z) = \frac{\eta(2z)^{3\alpha+1}}{\eta(z)^{3\alpha+2}}\cdot\eta(z)^{9A}\eta(2z)^{9r} E_4(z)^{9(A+r)}
    $$
    and
    $$
    g_2(z)=\frac{\eta(2z)^{3\alpha+1}}{\eta(z)^{3\alpha+2}}\cdot\eta(z)^{81A}\eta(2z)^{81r}
    $$
    are modular forms of weight $k=\frac{81(A+r)-1}{2}$, level $N = \frac{24}{\mathrm{gcd}(3A+\frac{-\alpha-1+3r}{2},8)}$, and character $$
\chi(d) = 
\begin{cases}
    \left(\frac{-1}{d}\right), &\text{ if }\quad \alpha\equiv 0,3 \pmod{4} \\
    1, &\text{ if }\quad \alpha \equiv 1,2 \pmod{4}.
\end{cases},
$$ assuming each has positive orders of vanishing at their respective cusps. Further, 
    $$
    g_1(z) \divides T_3^2 \equiv f_1^Af_2^r \sum_{n} a_{3\alpha+2}(9n+r)q^n \pmod{3}
    $$
    and
    $$
    g_2(z) \divides T_3^4 \equiv f_1^Af_2^r \sum_{n} a_{3\alpha+2}(81n+10t+9j)q^n \pmod{3},
    $$
where $T_3$ is the third Hecke operator.
\end{lem}

\begin{proof}
    It suffices to satisfy the conditions of Lemma \ref{1.64 1.65} in order to show the eta-quotients are modular forms. To that end, we require that 
    $$
    \alpha+3A+6r \equiv 0 \pmod{8}
    $$
    and
    $$
    \alpha+27A+54r \equiv 0 \pmod{8}.
    $$
    Clearly, the first congruence implies the second. Additionally, we need a positive integer $N$ such that 
    $$
    N\left(-(3\alpha+2)+\frac{3\alpha+1}{2}+9A+\frac{9}{2}r\right) \equiv 0 \pmod{24}
    $$
    and
    $$
    N\left(-(3\alpha+2)+\frac{3\alpha+1}{2}+81A+\frac{81}{2}r\right) \equiv 0 \pmod{24}.
    $$
   Again, the first congruence implies the second. Choosing $A \equiv -3t-3j-2r \pmod{24}$, as in the statement of the lemma, satisfies both congruences with $N = \frac{24}{\mathrm{gcd}(3A+\frac{-\alpha-1+3r}{2},8)}$. It is also straightforward to calculate that the character for each is given by 
$$
\chi(d) = 
\begin{cases}
    \left(\frac{-1}{d}\right), &\text{ if }\quad \alpha\equiv 0,3 \pmod{4} \\
    1, &\text{ if }\quad \alpha \equiv 1,2 \pmod{4}.
\end{cases}
$$ Thus we can conclude that $g_1(z), g_2(z) \in M_k(\Gamma_0(N),\chi)$ if they have positive orders of vanishing at the cusps.\\

Now, using the definition of $\eta(z)$ and the fact that $E_4 \equiv 1 \pmod{3}$, 
$$
g_1(z) \equiv f_9^Af_{18}^r\cdot\sum_{n\geq 0}a_{3\alpha+2}(n)q^{n+\frac{\alpha+3A+6r}{8}} \pmod 3,
$$
and
$$
g_1(z) \equiv f_{81}^Af_{162}^r\cdot\sum_{n\geq 0}a_{3\alpha+2}(n)q^{n+\frac{\alpha+27A+54r}{8}} \pmod 3.
$$
Hence, after a suitable reindexing,
$$
g_1(z) \divides T_3^2 \equiv f_1^Af_2^r \sum_{n} a_{3\alpha+2}\left(9n-\frac{\alpha+3A+6r}{8}\right)q^n \pmod{3},
$$
and
$$
g_2(z) \divides T_3^4 \equiv f_1^Af_2^r \sum_{n} a_{3\alpha+2}\left(81n-\frac{\alpha+27A+54r}{8}\right)q^n \pmod{3}.
$$
\end{proof}

We are now ready to prove our main theorem.

\begin{proof}[Proof of Theorem \ref{mainthm}]
    Let $\alpha \in \{1,3,4,6,7,9,11,12,14,15,20,22,23,27,28,30,31,36,38,39,46,47,54,55,63\}$. Note that Lemmas \ref{sellerslemma} and \ref{neccessarymodforms} provide the necessary conditions for Lemma \ref{mainlemma}. Additionally, for each $\alpha$, the orders of vanishing at the cusps for $g_1(z)$ and $g_2(z)$ have been verified with Maple to be positive. Hence, it suffices to check that $g_1(z) |T_3^2 \equiv g_2(z) | T_3^4 \pmod{3}$ up to the Sturm bound. This has been verified in Maple,
    which completes the proof.
\end{proof}

\section{Some Extra Congruences and Conclusion}

We conclude with some nice Ramanujan congruences for $a_{11}(n)$ and some questions. \\

Note that we have the following lemma.
\begin{lem}[Lemma 2.1 in \cite{hirschhorn2025familycongruencesmodulo7}]\label{Jlemma}
We have
    \begin{align*}
        f_1^3 &= \sum_{k \geq 0} (-1)^k(2k+1)q^{k(k+1)/2} \\
        &\equiv \mathcal{J}_0(q^7)+q\mathcal{J}_1(q^7) + q^3\mathcal{J}_3(q^7) \pmod{7},
    \end{align*}
    where $\mathcal{J}_0,\mathcal{J}_1,\mathcal{J}_3 \in \mathbb{Z}[[q]]$.
\end{lem}

We now prove the three Ramanujan congruences for $a_{11}$.

\begin{thm}
The following Ramanujan congruences hold,
\begin{align*}
    a_{11}(5n+4) &\equiv0 \pmod{5}, \\
    a_{11}(7n+4) &\equiv0 \pmod{7}, \\ 
   \text{ and } \quad a_{11}(11n+1) &\equiv0 \pmod{11}.
\end{align*}
\end{thm}

\begin{proof}
Note that 
    \begin{align*}
        \sum_{n\geq 0} a_{11}q^n &= \frac{f_2^{10}}{f_1^{11}} = \frac{f_2^{10}}{f_1^{10}} \cdot \frac{1}{f_1} \\
        &= \frac{f_2^{10}}{f_1^{10}} \cdot \sum_{n\geq 0}p(n)q^n
    \end{align*}
Since $\displaystyle\frac{f_2^{10}}{f_1^{11}}$ is a power series in $q^5$, it suffices to check that $p(5n+4) \equiv 0 \pmod{5}$. Thus the congruence follows from the Ramanujan congruence for $p(n)$. Also,

\begin{align*}
         \sum_{n\geq 0} a_{11}q^n &=\frac{f_2^{10}}{f_1^{11}} = \frac{f_2^7}{f_1^{14}} \cdot f_1^3f_2^3 \\
         &\equiv \frac{f_{14}}{f_7^{2}} \cdot (\mathcal{J}_0(q^7)+q\mathcal{J}_1(q^7)+q^3\mathcal{J}_3(q^7))\cdot(\mathcal{J}_0(q^{14})+q^2\mathcal{J}_1(q^{14})+q^6\mathcal{J}_3(q^{14})) \pmod{7},  \\
\end{align*}
    where the last line comes from Lemma \ref{Jlemma}. As $\displaystyle\frac{f_{14}}{f_7^{2}}$ is a power series in $q^7$, it suffices to check that the congruence is satisfied by the rest of the product. Note that in the above dissection, there is no combination that yields a term with exponent $7n+4$, proving the second congruence. Finally,
\begin{align*}
         \sum_{n\geq 0} a_{11}q^n &=\frac{f_2^{10}}{f_1^{11}} = \frac{f_2^{11}}{f_1^{11}f_2} \\
         &= \frac{f_2^{11}}{f_1^{11}} \sum_{n\geq 0} p(n)q^{2n}
\end{align*}
Again, as $\displaystyle\frac{f_2^{11}}{f_1^{11}}$ is a power series in $q^{11}$, it suffices to show the congruence holds for $\displaystyle\sum_{n\geq 0} p(n)q^{2n}$. Hence, we need to know when the exponent of $q$ in the sum is both $1 \pmod{11}$ and $0 \pmod{2}$. By the Chinese Remainder Theorem, this occurs when the exponent of $q$ is $12 \pmod{22}$, coinciding with $p(11n+6)$, which is the Ramanujan congruence mod 11 for $p(n)$.
\end{proof} 

It is likely the case that one could prove Ramanujan congruences for other $a_k$, $k \geq 0$. We conclude with a few questions that are of great interest to the author.
\begin{enumerate}
    \item Is there a nice explanation as to why some $\alpha < 81$ are missing from the list in Theorem \ref{mainthm}?
    \item Numerically, it would seem that the internal congruences offered by Lemma \ref{neccessarymodforms} stop around $\alpha = 70$. Is there a way to infinitely extend these congruence families for select $\alpha >0$?
    \item Do the Ramanujan congruences satisfied by $a_{11}(n)$ belong to an infinite family for other Ramanujan congruences satisfied by the various $a_{k}$'s?
\end{enumerate}

\newpage
\bibliographystyle{alpha}
\bibliography{refs}

\end{document}